\documentclass[11pt,letterpaper]{amsart}
\usepackage[parfill]{parskip}
\usepackage{amssymb}
\usepackage{amsmath}
\usepackage{amsthm}
\usepackage{amsfonts}
\usepackage{color}
\usepackage{graphicx, hyperref}
\usepackage{media9}
\usepackage[T1]{fontenc}

\usepackage{textcomp}
\usepackage{palatino}
\usepackage[text={6.5in, 8.5in}, centering]{geometry}

\newtheorem{theorem}{Theorem}

\theoremstyle{definition}

\theoremstyle{remark}

\begin{document}

\title{Local near-Beltrami structure and depletion of the nonlinearity in the 3D Navier-Stokes flows}

\author{Aseel Farhat}
\address{Department of Mathematics, Florida State University}
\email{afarhat@fsu.edu}

\author{Zoran Gruji\'c}
\address{Department of Mathematics, University of Virginia}
\email{zg7c@virginia.edu}

\date{\today}

\begin{abstract}
Computational simulations of turbulent flows indicate that the regions of low dissipation/enstrophy production feature high degree of local alignment between the velocity and the vorticty, i.e., the flow is
locally near-Beltrami. Hence one could envision a geometric scenario in which the persistence of the local near-Beltramy property might be consistent with a (possible) finite-time singularity
formation. The
goal of this note is to show that this scenario is in fact prohibited if the sine of the angle between the velocity and the vorticty is small enough
with respect to the local enstrophy.
\end{abstract}

\maketitle

\section{Introduction}

The motion of 3D incompressible, viscous, Newtonian fluid is described by the 3D Navier-Stokes (NS) equations,
\[
 u_t + (u \cdot \nabla) u = \nu \triangle u + \nabla p + f, \ \ \ \mbox{div} \, u = 0
\]
where $u$ is the velocity of the fluid, $p$ the pressure, $\nu$ the viscosity, and $f$ the external force. This is
supplemented with the initial condition $u_0$ and the appropriate boundary conditions. For simplicity, set the
viscosity to be 1, and suppose that the external force is of the potential type.

\medskip

Since the pioneering work of Leray \cite{Leray3} in 1930s, it has been known that if the initial condition is regular enough, 
the solution remains regular at least over a finite interval of time, the length 
of which depends on a suitable norm of the initial condition. The question of whether this time-interval can be extended 
to infinity for arbitrary large initial data remains open, and is usually referred to as the NS regularity problem.

\medskip

Taking the curl of the velocity-pressure formulation leads to the vorticity-velocity formulation of the 3D NS equations,
\[
 \omega_t + (u \cdot \nabla) \omega = \triangle \omega + (\omega \cdot \nabla) u;
\]
here, $\omega = \, \mbox{curl} \, u$ is the vorticity, the left-hand side represents transport of the vorticity by the velocity, 
the first term on the right-hand side
is the diffusion, and the second one the vortex-stretching term (which is absent in the 2D case where it is a routine exercise
to show that the flow remains regular for all times).

\medskip

Geometric studies of the NS regularity problem were initiated by Constantin in \cite{Co94} where he presented a singular
integral representation of the stretching factor in the evolution of the vorticity magnitude featuring a geometric kernel 
that was depleted by the local coherence of the vorticity direction. This led to a fundamental result \cite{CoFe93} stating that
as long as the vorticity direction remains Lipschitz (in the regions of intense vorticity) the flow remains regular. 
Subsequently--among the other results--the Lipchitz condition was scaled down to $\frac{1}{2}$-H\"older \cite{daVeigaBe02} and the 
spatiotemporal localization of various geometric, analytic and hybrid geometric-analytic regularity criteria was
presented in \cite{Gr09, GrGu10-1}. Let us remark that the local coherence of the vorticity direction is--in a 
sense--a locally near-2D property.

\medskip

The aforementioned approach was inspired by the computational simulations of turbulent flows revealing that the regions
of intense vorticity are dominated by coherent vortex structures and in particular vortex tubes/filaments. This type of geometry 
plausibly exhibits local coherence of the vorticity direction field. However, the simulations also
indicate the formation of local structures in which the vortex lines `meet' transversely or near-transversely
(cf. \cite{SJO91, HK02}) hinting at the possibility of a discontinuity in the vorticity direction field. 
These structures--at the same time--feature a local near-alignment
of the velocity and the vorticity, i.e., they are locally near-Beltrami. Note that being near-Beltrami is not a near-2D property. 

\medskip

The analysis of the probability density function for the distribution of the angle between velocity and vorticity conditionally sampled in the regions of high and low dissipation given in \cite{PYO85} revealed the tendency to evenly distribute and align, respectively (in particular, this was the case in Taylor-Green vortex flows). In addition, the statistical analysis of the data sourced from a direct numerical simulation of forced isotropic turbulence performed in \cite{CKL09} demonstrated a strong correlation between the local helicity and the local enstrophy. Hence, on one hand, there are indications that the local near-Beltrami structures are a signature of the regions of intense fluid activity, and persistence of the local near-Beltrami property might be consistent with possible formation of a singularity. On the other hand, the exact Beltrami flows are essentially linear (the vorticity satisfies the 3D heat equation), and it is natural to wonder whether it would be possible to effectively analyze the \emph{locally near-Beltrami flows} as the locally `near-linear' flows.

\medskip

In the special case of the \emph{globally} near-Beltrami flow, it was shown in \cite{daVeiga12} (in the case of the stress-free boundary conditions) that as long as the sine of the angle between the velocity and the vorticity is \emph{uniformly small}  with respect to the \emph{global enstophy}, the solution remains regular.
In this note, we demonstrate that the geometric depletion of the nonlinearity induced by the local near-Beltrami structure is effective in the setting of the \emph{spatiotemporal localization} of the flow to an arbitrarily small parabolic cylinder. More precisely, we show that if the sine of the angle between the velocity and the vorticity is small enough (on a local parabolic cylinder) with respect to the 
\emph{local enstrophy}, no singularity can form.

\medskip   

Several related results on the regularity issue for the 3D NSE using the Beltrami property are proved in \cite{BC09} and \cite{Ch10}. In \cite{BC09}, the authors show that if the velocity and the vorticity are near-orthogonal (the helicity is nearly-zero) at each point, then the solution of the 3D NSE is smooth. The flow in this case is "anti-Beltrami" and resembles the 2D geometric situation. In \cite{Ch10}, regularity criteria for 3D NSE are proved under a number of different conditions on the helicity vector $u\times \omega$ including boundedness in $L^{3,\infty}$.

\medskip 

At the end, let us mention that there are results in the literature in which suitably defined helical invariances and helical projections come to play. In particular, in \cite{MTL90} the authors obtained the global regularity for the helically symmetric solutions to the 3D NS equations, while \cite{BT13} demonstrated the global regularity for solutions living in the subspace of a sign-definite helicity.

\section{The main result}

The standard vorticity-velocity formulation of the 3D NS equations recalled in Introduction can be written in the form suitable for the study of local helicity, namely, 
\begin{align}\label{vor}
\omega_t -\Delta \omega + \nabla\times(\omega\times u) =0.
\end{align}

\medskip

Let $(x_0,t_0)$ be a point in the space-time, and let $Q_r(x_0,t_0)$ be the parabolic cylinder $B(x_0,r) \times (t_0-r^2, t_0)$. Define $\psi(x,t) = \phi(x)\eta(t)$ to be a spatiotemporal  cut-off function such that: $0\leq \phi\leq 1$ is supported in $B(x_0,2r)$ with 
\[ \phi(x) = 1, \quad x\in B(x_0,r); \quad \frac{|\nabla \phi|}{\phi^\delta} \leq \frac{C}{r}, \]
for some $\delta \in (0,1)$, and $0\leq \eta\leq1$ is supported in $(t_0-(2r)^2, t_0]$ with 
\[ \eta(t) = 1, \quad t\in [t_0-r^2, t_0]; \quad |\eta^{'}|\leq \frac{2}{r^2}. \]
Denote by $Q_{2r}^t(x_0,t_0)$ the parabolic sub-cylinder of $Q_{2r}(x_0,t_0)$, $Q_{2r}^t (x_0,t_0)= B(x_0,2r) \times (t_0-(2r)^2, t),$ where $t\in (t_0-(2r)^2, t_0)$. 

For $(x,t) \in Q_{2r}$, define $$\alpha(t) = \sup_{x\in B(x_0,2r)} \sin\left(\theta(x,t)\right)$$
where
$$\theta(x,t) = \angle (\omega(x,t), u(x,t)).$$

For a (large) number $M>0$ define 
$$S = \{ (y,s): |\omega(y,s)|>M\} \cap Q_{2r}, $$ and 
$$ S^t = \{ (y,s): |\omega(y,s)|>M\} \cap Q_{2r}^t$$ 
for every $t\in (t_0-(2r)^2, t_0)$. 
In addition, define
$$S_s = \{ y: |\omega(y,s)|>M\} \cap B(x_0,2r), $$ 
$$S_y = \{ s: |\omega(y,s)|>M\} \cap (t_0-(2r)^2, t_0), $$ and 
$$S_y^t = \{ s: |\omega(y,s)|>M\} \cap  (t_0-(2r)^2, t) $$  for every $t\in (t_0-(2r)^2, t_0)$. 

\begin{theorem} 
Let $u$ be a Leray-Hopf solution of the 3D NSE in $\mathbb{R}^3\times (0,\infty)$ with initial data $u_0\in L^2(\mathbb{R}^3)$. Given $(x_0,t_0)\in \mathbb{R}^3\times (0,\infty)$ and $r>0$, suppose that the flow
is smooth in the open cylinder $Q_{2r}(x_0,t_0)$, up to its parabolic boundary, and assume that
there exists a constant $c$ such that 
$$ \alpha(s)\|\nabla u(s)\|_{L^2(S_s)}^{1/2} \leq c$$
for every $s\in (t_0-(2r)^2, t_0)$. Then, the localized enstrophy remains bounded in $Q_{r}(x_0,t_0)$, i.e. 
$$\sup_{t\in(t_0-r^2,t_0)} \int_{B(x_0,r)} |\omega|^2(x,t)\, dx  < \infty$$ and--consequently--$(x_0,t_0)$ is a 
regular point.
\end{theorem} 

\begin{proof} 
Multiply equation \eqref{vor} by $\psi^2\omega$ and integrate over the sub-cylinder $Q_{2r}^t$,
\begin{align}\label{major_1}
\int_{Q_{2r}^t} \omega_s\cdot (\psi^2\omega) \,dy\,ds -\int_{Q_{2r}^t} \Delta\omega\cdot (\psi^2\omega ) \,dy \,ds =&  - \int_{Q_{2r}^t} \nabla\times(\omega\times u)\cdot(\psi^2 \omega) \,dy \,ds. 
\end{align}

Integration by parts yields 
\begin{align}\label{est1_1}
-\int_{Q_{2r}^t} \Delta \omega\cdot (\omega \psi^2) \,dy \,ds & = \int_{Q_{2r}^t} \nabla\omega \cdot \nabla(\psi^2\omega)\, dy\,ds \notag \\
& = \int_{Q_{2r}^t} \nabla\omega \cdot \nabla[(\psi\omega)\psi]\, dy\, ds\notag \\
& = \int_{Q_{2r}^t} [\nabla(\psi\omega) \cdot (\psi\nabla \omega) + (\psi\omega) \cdot (\nabla\psi \nabla\omega)]\, dy \,ds\notag\\
& = \int_{Q_{2r}^t} [\nabla(\psi\omega) \cdot (\nabla(\psi\omega) - (\nabla\psi) \omega) +(\psi\omega) \cdot (\nabla\psi \nabla\omega)]\, dy \,ds\notag\\
& = \int_{Q_{2r}^t} [|\nabla(\psi \omega)|^2 - |\nabla\psi|^2|\omega|^2]\, dy \,ds.
\end{align}

Notice that 
\begin{align}\label{est2_1}
\int_{Q_{2r}^t} \omega_s\cdot(\omega \psi^2)\,dy\,ds &= \frac12\int_{Q_{2r}^t} |\omega|^2_s \psi^2\,dy\,ds \notag\\
& = -\int_{Q_{2r}^t} |\omega|^2\phi^2 \eta\eta_s\,dy\,ds + \frac12 \int_{B(x_0,2r)} |\omega|^2(y,t)\phi^2(y,t)\, dy. 
\end{align}


At this point, it is beneficial to recall the following identity,
\begin{align*}
\int_\Omega (\nabla\times f)\cdot g\,dx = \int_\Omega f\cdot(\nabla \times g)\,dx + \int_\Gamma(n\times f)\cdot g\,d\Gamma. 
\end{align*}
Since $\phi =0$ on $\partial B(x_0,2r)$, 
\begin{align}\label{est3_1}
\int_{Q_{2r}^t}\nabla\times(\omega \times u) \cdot (\psi^2\omega)\,dy\,ds &= \int_{t_0-(2r)^2}^t \int_{B(x_0,2r)}\nabla\times(\omega \times u)\cdot (\psi^2\omega)\,dy\,ds\notag \\
& = \int_{t_0-(2r)^2}^t \int_{B(x_0,2r)} (\omega\times u) \cdot \left(\nabla\times(\psi^2\omega)\right)\,dy\,ds\notag \\
& = \int_{Q_{2r}^t} (\omega\times u) \cdot\left(\nabla\times[\psi (\psi\omega)]\right)\,dy\,ds \notag \\
& =\int_{Q_{2r}^t} \psi(\omega\times u) \cdot [\nabla\times(\psi\omega)]\,dy\,ds +\int_{Q_{2r}^t} (\omega\times u) \cdot[\nabla\psi \times(\psi\omega)]\,dy\,ds, 
\end{align}
where we used the standard formula in vector analysis: 
$$ \nabla\times(\phi {\bf A}) = \phi(\nabla \times {\bf A})+(\nabla\phi)\times {\bf A}.$$

Combining \eqref{major_1} with \eqref{est1_1}--\eqref{est3_1}, we arrive at the following
\begin{align}\label{major2_1}
& \int_{Q_{2r}^t} |\nabla(\psi \omega)|^2\,dy\,ds + \frac12 \int_{B(x_0,2r)} |\omega|^2(y,t)\phi^2(y,t)\, dy \notag\\
& \qquad  \qquad \leq  \int_{Q_{2r}^t} |\omega|^2\left(\phi^2 |\eta||\eta_s| + |\nabla\psi|^2|\right)\,dy\,ds \notag \\
& \qquad \qquad  \quad - \int_{Q_{2r}^t} \psi(\omega\times u) \cdot [\nabla\times(\psi\omega)]\,dy\,ds - \int_{Q_{2r}^t} (\omega\times u) \cdot[\nabla\psi \times(\psi\omega)]\,dy\,ds\notag\\
& \qquad \qquad \equiv I_1- I_2 - I_3. 
\end{align} 

It is plain that 
\begin{align}\label{est4_1}
I_1 &\leq \frac{c}{r^2}  \int_{Q_{2r}^t} |\omega|^2\,dy\,ds \notag\\
&\leq  \frac{c}{r^2}  \int_{Q_{2r}^t} |\nabla u|^2\,dy\,ds
\end{align}
where $c$ denotes a generic constant that may change from line to line. 

Splitting the local flow into the regions of low and high vorticity,
\begin{align*}
I_2 &= \int_{Q_{2r}^t\cap \{|\omega|\leq M\}} \psi(\omega\times u) \cdot [\nabla\times(\psi\omega)]\,dy\,ds+ \int_{Q_{2r}^t\cap \{|\omega|> M\}} \psi(\omega\times u) \cdot [\nabla\times(\psi\omega)]\,dy\,ds\\
& \equiv I_2^{'} + I_2^{''},
\end{align*} 
and 
\begin{align*}
I_3 &= \int_{Q_{2r}^t\cap \{|\omega|\leq M\}} 
(\omega\times u ) \cdot[\nabla\psi \times(\psi\omega)]\,dy\,ds +  \int_{Q_{2r}^t\cap \{|\omega|>M\}} (\omega\times u) \cdot[\nabla\psi \times(\psi\omega)]\,dy\,ds\\
& \equiv I_3^{'} + I_3^{''}. 
\end{align*}

In the regions of low vorticity,
\begin{align}\label{est5_1}
|I_2^{'}| & \leq \frac12 \int_{Q_{2r}^t\cap \{|\omega|\leq M\}} \psi^2 |u|^2|\omega|^2 \,dy\,ds + \frac12\int_{Q_{2r}^t\cap \{|\omega|\leq M\}}|\nabla(\psi\omega)|^2\,dy\,ds\notag\\
& \leq 2M^2r^2 \|u_0\|_{L^2(B(x_0,2r))}^2+ \frac12\int_{Q_{2r}^t}|\nabla(\psi\omega)|^2\,dy\,ds, 
\end{align} 
and 
\begin{align}\label{est6_1}
|I_3^{'}| & \leq \|\nabla \psi\|_{L^\infty(Q_{2r}^t)} M^2 \int_{Q_{2r}^t} |u|\,dy\,ds\notag\\
& \leq \frac{cM^2}{r} \int_{t_0-(2r)^2}^t \int_{B(x_0,2r)} |u|\,dy\,ds \notag\\
& \leq \frac{cM^2}{r} \times 4r^2 \times \left(\frac43 (2r)^3 \pi\right)^{1/2} \times \|u_0\|_{L^2(B(x_0,2r))}\notag \\
& \leq c r^{5/2}\|u_0\|_{L^2(B(x_0,2r))}. 
\end{align}

Next, we estimate the nonlinear terms in the regions of high vorticity. Start with $I_2^{''}$.
\begin{align*}
|I_2^{''}| & = \left | \int_{Q_{2r}^t\cap \{|\omega|>M\}} [(\psi\omega)\times u]\cdot\nabla\times(\psi\omega)\,dy\,ds\right |\notag\\
& \leq \int_{Q_{2r}^t\cap \{|\omega|>M\}} \alpha(s)|u||\psi\omega||\nabla\times(\psi\omega)|\,dy\,ds.
\end{align*}
Using H\"older's inequality w.r.t. $y$, Sobolev embedding and Ladyzhenskaya's inequality gives
\begin{align*}
|I_2^{''}|& \leq \int_{t_0-(2r)^2}^t \alpha(s) \|u\|_{L^6(S_s)} \|\psi\omega\|_{L^3(B(x_0,2r))} \|\nabla(\psi\omega)\|_{L^2(B(x_0,2r))}\,ds \notag \\
&\leq c  \int_{t_0-(2r)^2}^t \alpha(s)\|\nabla u\|_{L^2(S_s)} \|\psi\omega\|_{L^2(B(x_0,2r))}^{1/2} \|\nabla(\psi\omega)\|_{L^2(B(x_0,2r))}^{3/2}\,ds,
\end{align*}

and utilizing H\"older's inequality w.r.t. $s$, followed by Young's inequality yields
\begin{align}\label{est7_1}
|I_2^{''}| &\leq c \left( \int_{t_0-(2r)^2}^t \alpha^4(s) \|\psi\omega\|_{L^2(B(x_0,2r)}^2 \|\nabla u\|_{L^2(S_s)}^4\, ds \right)^{1/4}\left( \int_{t_0-(2r)^2}^t \|\nabla (\psi\omega)\|_{L^2(B(x_0,2r))}^2\,ds\right)^{3/4}\notag\\
&\leq c \sup_{s\in(t_0-(2r)^2,t)} \left(\alpha(s) \|\psi\omega(s)\|_{L^2(B(x_0,2r))}^{1/2} \|\nabla u(s)\|_{L^2(S_s)}^{1/2}\right)\notag\\
& \qquad \qquad \qquad \qquad \times \left(\int_{t_0-(2r)^2}^t \|\nabla u\|_{L^2(S_s)}^2\,ds\right)^{1/4} \|\nabla (\psi\omega)\|_{L^2(Q_{2r}^t)}^{3/2}\notag \\
& \leq c \sup_{s\in(t_0-(2r)^2,t)} \left(\alpha(s) \|\nabla u(s)\|_{L^2(S_s)}^{1/2}\right)\|\nabla u\|_{L^2(Q_{2r}^t)}^{1/2} \notag \\
& \qquad \qquad \qquad \qquad \times \left( \frac34 \|\nabla(\psi\omega)\|_{L^2(Q_{2r}^t)}^{2} + \frac14 \sup_{s\in(t_0-(2r)^2,t)} \|\psi\omega(s)\|_{L^2(B(x_0,2r))}^{2}\right).
\end{align}

Applying a similar argument as above, H\"older's inequality w.r.t $y$ , Sobolev embedding and Ladyzhenskaya's inequality yield 
\begin{align*}
|I_3^{''}| &=\left | \int_{Q_{2r}^t\cap \{|\omega|>M\}} (\omega\times u) \cdot[\nabla\psi \times(\psi\omega)]\,dy\,ds\right |\notag\\
& \leq \int_{Q_{2r}^t\cap \{|\omega|>M\}} \alpha(s)|\nabla \psi||\omega||u||\psi\omega|\,dy\,ds\notag\\
& \leq \int_{Q_{2r}^t\cap \{|\omega|>M\}} \alpha(s)|\nabla \phi \eta||\omega||u||\psi\omega|\,dy\,ds\notag\\
& \leq Cr^{-1}\int_{Q_{2r}^t\cap \{|\omega|>M\}} \alpha(s)|\omega||u||\psi\omega|\,dy\,ds\notag\\
&\leq Cr^{-1} \int_{t_0-(2r)^2}^t \alpha(s) \|\omega\|_{L^2(B(x_0,2r))}\|u\|_{L^3(S_s)} \|\psi\omega\|_{L^6(B(x_0,2r))}\, ds\\
&\leq cr^{-1} \int_{t_0-(2r)^2}^t \alpha(s) \|\omega\|_{L^2(S_s)} \|u\|_{L^2(B(x_0,2r))}^{1/2}\|\nabla u\|_{L^2(S_s)}^{1/2}\|\nabla (\psi\omega)\|_{L^2(B(x_0,2r))}\,ds\\
&\leq cr^{-1} \|u_0\|_{L^2(B(x_0,2r))}^{1/2}\int_{t_0-(2r)^2}^t \alpha(s) \|\nabla u\|_{L^2(S_s)}^{3/2}\|\nabla (\psi\omega)\|_{L^2(B(x_0,2r))}\,ds\\
& \leq cr^{-1}\|u_0\|_{L^2(B(x_0,2r))}^{1/2} \sup_{s\in(t_0-(2r)^2,t)} \left(\alpha(s) \|\nabla u(s)\|_{L^2(S_s)}^{1/2}\right)\\
& \qquad \qquad \qquad \qquad \times\int_{t_0-(2r)^2}^t \|\nabla u\|_{L^2(S_s)}\|\nabla (\psi\omega)\|_{L^2(B(x_0,2r))}\,ds, 
\end{align*}

while H\"older's inequality w.r.t. $s$ in conjunction with Young's inequality imply 
\begin{align}\label{est8_1}
|I_3^{''}| & \leq cr^{-1}\|u_0\|_{L^2(B(x_0,2r))}^{1/2} \sup_{s\in(t_0-(2r)^2,t)} \left(\alpha(s) \|\nabla u(s)\|_{L^2(S_s)}^{1/2}\right)\|\nabla u\|_{L^2(Q_{2r}^t)} \|\nabla (\psi\omega)\|_{L^2(Q_{2r}^t)} \notag \\
&\leq cr^{-2}\|u_0\|_{L^2(B(x_0,2r))}\sup_{s\in(t_0-(2r)^2,t)} \left(\alpha(s) \|\nabla u(s)\|_{L^2(S_s)}^{1/2}\right)^2\|\nabla u\|_{L^2(Q_{2r}^t)}^2 + \frac14 \|\nabla (\psi\omega)\|_{L^2(Q_{2r}^t)} ^2. 
\end{align}

Combining the estimates \eqref{est4_1}--\eqref{est8_1} with \eqref{major2_1}, we get the following bound
\begin{align}\label{major3_1}
& \frac{1}{2}\int_{Q_{2r}^t} |\nabla(\psi \omega)|^2\,dy\,ds + \frac12 \int_{B(x_0,2r)} |\omega|^2(y,t)\phi^2(y,t)\, dy \notag\\
& \qquad \leq c\left(r, M, \|u_0\|_{L^2(B(x_0,2r))}, \|\nabla u\|_{L^2(Q_{2r}^t)}, \sup_{s\in(t_0-(2r)^2,t)} \left(\alpha(s) \|\nabla u(s)\|_{L^2(S_s)}^{1/2}\right)\right) \notag \\
& \qquad  \quad + c\sup_{s\in(t_0-(2r)^2,t)}\left(\alpha(s)\|\nabla u(s)\|_{L^2(B(x_0,2r))}^{1/2}\right)\|\nabla u\|_{L^2(Q_{2r}^t)}^{1/2}\notag \\
& \qquad  \qquad \qquad \qquad \qquad \times \left(\|\nabla (\psi\omega)\|_{L^2(Q_{2r}^t)}^2 + \sup_{s\in(t_0-(2r)^2,t)} \|\psi\omega(s)\|_{L^2(B(x_0,2r))}^{2}\right),
\end{align}
for every $t\in(t_0)-(2r)^2, t_0)$. Since
$$ \sup_{s\in(t_0-(2r)^2,t)}\left(\alpha(s)\|\nabla u(s)\|_{L^2(S_s)}^{1/2}\right) \leq c, $$
for every $t\in(t_0-(2r)^2,t_0)$, taking the supremum in $t\in(t_0-(2r)^2,t_0)$ in \eqref{major3_1} yields 
\begin{align*}
& \int_{Q_{2r}} |\nabla(\psi \omega)|^2\,dy\,ds + \sup_{t\in(t_0-(2r)^2,t_0)} \int_{B(x_0,2r)} |\omega|^2(y,t)\phi^2(y,t)\, dy \notag\\
& \qquad \leq c\left(r, M, \|u_0\|_{L^2(B(x_0,2r))}, \|\nabla u\|_{L^2(Q_{2r})}\right) \notag \\
& \qquad + c\|\nabla u\|_{L^2(Q_{2r})}^{1/2}\left(\|\nabla (\psi\omega)\|_{L^2(Q_{2r}^t)}^2 + \sup_{t\in(t_0-(2r)^2,t_0)} \|\phi\omega(s)\|_{L^2(B(x_0,2r))}^{2}\right).
\end{align*}

Since $u$ is a Leray solution, $\|\nabla u\|_{L^2(Q_{2r})} <\infty$, and
$$c\left(r, M, \|u_0\|_{L^2(B(x_0,2r))}, \|\nabla u\|_{L^2(Q_{2r})}\right) <\infty.$$
Moreover, $\|\nabla u\|_{L^2(Q_{2r})} \, \rightarrow \, 0$ as $r\, \rightarrow \, 0$; hence, for every $\varepsilon >0$ there exists $\rho(\varepsilon)>0$ such that $\|\nabla u\|_{L^2(Q_{2r})} \leq \varepsilon$ for every $r\leq \rho$. 
Setting  $\varepsilon = \frac{1}{4c^2}$ yields
\begin{align*}
& \int_{Q_{2r}} |\nabla(\psi \omega)|^2\,dy\,ds +  \sup_{t\in(t_0-(2r)^2,t_0)}  \int_{B(x_0,2r)} |\omega|^2(y,t)\phi^2(y,t)\, dy \notag\\
& \qquad \leq c\left(r, M, \|u_0\|_{L^2(B(x_0,2r))}, \|\nabla u\|_{L^2(Q_{2r})}\right),
\end{align*}
for any $r \le \rho =\rho(\varepsilon)$.
Consequently,
\begin{align}\label{final_1}
 \sup_{t\in(t_0-(2r)^2,t_0)}\int_{B(x_0,r)} |\omega|^2(y,t)\, dy & \leq \sup_{t\in(t_0-(2r)^2,t_0)} \int_{B(x_0,2r)} |\omega|^2(y,t)\phi^2(y,t)\, dy\notag \\
 &  \leq c\left(r, M, \|u_0\|_{L^2(B(x_0,2r))}, \|\nabla u\|_{L^2(Q_{2r})}\right) < \infty,
\end{align}
for any $r\leq \rho$. If $r>\rho$, covering $B(x_0,r)$ with finitely many balls $B(z_0,\rho)$ for suitable $z_0$s 
and redoing the proof on each cylinder $B(z_0,2\rho)\times (t_0-(2\rho)^2, t)$ yields the desired bound. 

\end{proof}

\bigskip

\centerline{\textbf{Acknowledgments}}

\medskip

The work
of Z.G. is supported in part by the National Science Foundation grant DMS - 1515805.

\end{document}